\documentclass{amsart}[12]
\usepackage{mathrsfs, amsmath,amsfonts,amssymb}
\usepackage{amsmath}
\usepackage{mathtools}
\usepackage{fullpage}

\theoremstyle{plain}

\newtheorem{corollary}{Corollary}

\newtheorem{lem}{Lemma}

\newtheorem{prop}{Proposition}
\newtheorem{remark}{Remark}

\newtheorem{thm}{Theorem}
\numberwithin{equation}{section}

\title{Laplace-type integral representations of the generalized Bessel function and of the Dunkl kernel of type $B_2$}

\author[B. Amri]{B\'echir Amri}
\address{Universit\'e Tunis El Manar, Facult\'e des sciences de Tunis \\ Laboratoire d'Analyse Math\'ematique et Applications, LR11ES11 
\\ 2092 El Manar I, Tunisie}
\email{bechir.amri@gmail.com}

\author[N. Demni]{Nizar Demni}
\address{IRMAR, Universit\'e de Rennes 1\\ Campus de
Beaulieu\\ 35042 Rennes cedex\\ France}
\email{nizar.demni@univ-rennes1.fr}

\begin{document}
\maketitle
\begin{abstract}
ln this paper, we derive a Laplace-type integral representations for both the generalized Bessel function and the Dunkl kernel associated with the rank-two root system of type $B_2$. 
The derivation of the first one elaborates on the integral representation of the generalized Bessel function proved in \cite{Demni} through the modified Bessel function of the first kind. In particular, we recover an expression of the density of the Duistermaat-Heckman measure for the dihedral group of order eight. As to the integral representation of the corresponding Dunkl kernel, it follows from an application of the shift principle to the generalized Bessel function.  
\end{abstract}

\section{Introduction}
Dunkl operators were introduced by C. F. Dunkl in the late eightees and form a commutative algebra of differential-difference operators (\cite{Dunkl}). They give rise to the so-called Dunkl kernel which extends the exponential kernel to finite reflection groups and multiplicity functions, and served together with their trigonometric analogues to prove the integrability of Calogero-Moser-Sutherland systems (see e.g. \cite{CDGRVY}, Ch.I). They also allow for a rich harmonic analysis which reduces for special multiplicity values to the one on Cartan motion groups associated with semi-simple Lie groups (\cite{CDGRVY}, Ch.I). The key tool of this analysis is the so-called generalized Bessel function (hereafter GBF) which generalizes the spherical functions on Euclidean-type symmetric spaces. At the probabilistic side, the Dunkl-Laplace operator is the infinitesimal generator of a Feller process referred to as the Dunkl process (\cite{CDGRVY}, CH. II and III). The latter as well as its projection on the positive Weyl chamber were extensively studied and exhibits very interesting properties extending those satisfied by the Brownian motion on Euclidean spaces or reflected in Weyl chambers, and by Bessel processes (see also \cite{Dem0}, \cite{Dem-Lep}). In particular, the definition of the Duistermaat-Heckman measure associated with a co-adjoint orbit of a compact semi simple Lie group makes now sense for any finite Coxeter group (\cite{BBO}). Moreover, this measure encodes the Laplace-type integral representation of the GBF when all the multiplicity values equal one. For arbitrary multiplicity values with nonnegative real part, a Laplace-type integral representation still holds and the representative measure is supported in a convex polytope (\cite{dJ}, Corollary 3.3). For $A$-type root systems, its explicit description is obtained in \cite{Amri0} using a contraction principle and induction. There, the polytope supporting the representative measure is the image of the top-row of the $A$-type Gelfand-Tsetlin graph by an affine transformation. Note that this graph as well as its analogues of types $B,C,D$ arise in the study of radial parts of minors of Hermitian random matrices over division algebras (see \cite{Def}). However, beware that only the eigenvalues of complex Hermitian ones considered in \cite{Def} fit into the theory of Dunkl operators. On the other hand, the Dunkl kernel is also positive-definite (\cite{Ros}) and few attempts to derive its Laplace-type integral representation have been done in \cite{Dunkl1}, \cite{Dunkl2}, \cite{Amri1}. Both papers \cite{Dunkl1} and \cite{Amri1} deal with the $A_2$-type Dunkl kernel where its integral representation is deduced from that of the corresponding GBF using the shift principle. Nonetheless, the latter improves the former since it is valid for a larger set of multiplicity values and has less integration variables. As to \cite{Dunkl2}, it deals with the $B_2$-type root system and equal multiplicity values. There, the strategy of derivation of the integral representation of the Dunkl kernel is very similar to the one followed in \cite{Dunkl1} and the obtained formula involves five integration variables. 

In this paper, we derive a Laplace-type integral representation of the GBF of type $B_2$ over $\mathbb{R}^2$ valid for arbitrary multiplicity values whose sum is greater than $1/2$. Our derivation starts from the integral representation of the GBF obtained in \cite{Demni} and is far from being a direct consequence of it. Actually, it relies on a tricky way of writing a six-variables polynomial as the squared Euclidean norm of a planar vector together with an integral representation of the modified Bessel function we prove below. Moreover, the representative measure is shown to be supported in the convex hull of either of both variables as it is supposed to be by the virtue of Corollary 3.3 in \cite{dJ}. When both multiplicity values are equal to one, we get an explicit expression of the Duistermaat-Heckman measure associated with the dihedral group of order eight. Afterwards, we use the shift principle to derive a Laplace-type integral representation for the Dunkl kernel of type $B_2$, which is technical and a bit lengthy. Nonetheless, we believe  that our resulting formula is optimal in the sense that it as most compact as possible.   
 
The paper is organized as follows. In the next section, we recall some facts on Dunkl operators and the integral representation of the GBF of type $B_2$ proved in \cite{Demni}. In the third section, we derive its Laplace-type integral representation and deduce the  
Duistermaat-Heckman measure associated with the dihedral group of order eight. The fourth section is devoted to the application of the shift principle and to the technical derivation of the Laplace-type integral representation of the Dunkl kernel of type $B_2$. 

\section{Root systems, the Dunkl kernel and the generalized Bessel function} 
For facts on root systems and on Dunkl operators, we refer the reader to the monograph \cite{Dun-Xu}. Let $(V,\langle \cdot, \cdot \rangle)$ be a Euclidean space of dimension $N$ and denote by $\mathopen\|\cdot\mathclose\|:=\langle \cdot, \cdot \rangle^{1/2}$ the corresponding Euclidean norm. A root system $R$ in $V$ is a finite set of vectors (roots) in $V \setminus \{0\}$ such that 
\begin{equation*}
\forall \alpha \in R, \quad \sigma_{\alpha}(R) = R, 
\end{equation*}
where $\sigma_{\alpha}$ is the reflection with respect to $\alpha^{\perp}$: 
\begin{equation*}
\forall x \in V, \quad \sigma_{\alpha}(x) = x - 2\frac{\langle \alpha, x\rangle}{\langle \alpha, \alpha\rangle}\alpha. 
\end{equation*} 
It is reduced if 
\begin{equation*}
\forall \alpha \in R, \quad R \cap \mathbb{R} \alpha = \{\pm \alpha\}. 
\end{equation*}
The set $\{\sigma_{\alpha}, \alpha \in R\}$ generates a finite group $W$, called the reflection group associated with $R$, which acts on $R$ in a natural way. With this action in hands, a $W$-invariant function  $k:R  \to  \mathbb{C}$ is called a multiplicity function. Therefore, $k$ takes as many values as the number of orbits of $R$ under the action of $W$. Now, we can split $R$ into a positive and a negative parts $R_+$ and $R_-$ by choosing a non zero vector $v \in V \setminus R$ and to every choice of positive system $R_+$, there exists a unique set $S$ of indecomposable roots referred to as the simple system. The elements of $S$ are the so-called simple roots and characterize $R_+$ as the set of roots that are linear combinations of simple roots with non negative coefficients. 

Now, given a data $(R,R_+,k),$ the Dunkl operator or derivative $T_{\xi}$ in the direction $\xi \in V \setminus \{0_V\}$ is defined by: 
\begin{equation*}
T_{\xi}(f)(x) = \partial_{\xi}(f)(x) + \sum_{\alpha \in R_+} k(\alpha) \langle \alpha, \xi \rangle \frac{f(x) - f(\sigma_{\alpha}x)}{\langle \alpha,x\rangle}. 
\end{equation*}
The Dunkl operators form a commutative algebra and if $(e_1,\dots, e_N)$ is the canonical basis of $V$, then we shall simply write $T_i$ for $T_{e_i}, 1 \leq i \leq N$. Moreover, for generic values of $k$ (for instance those with non negative real parts), there exists a kernel 
\begin{equation*}
(x,y) \in V \times V \mapsto D_k(x,y)  
\end{equation*}
named the Dunkl kernel, which is the unique analytic solution of the spectral problem: 
\begin{eqnarray*}
T_{\xi}D_k(\cdot, y)(x) &=& \langle \xi, y\rangle D_k(x,y), \quad \textrm{for all} \quad \xi \in V \in \setminus \{0_V\}, \\ 
D_k(0_V,y) & = & 1, \quad \textrm{for all} \quad y \in V. 
\end{eqnarray*}
This kernel is even positive definite: 
\begin{equation*}
D_k(x,y) = \int e^{\langle x, z\rangle} \mu_{y}^{(k)}(dz) 
\end{equation*}
for some probability measure $\mu_{y}^{(k)}$ supported in the convex hull $\textrm{co}(y)$ of the orbit $Wy$. More generally, this measure defines the so-called Dunkl intertwining operator which intertwines the algebras of Dunkl and of partial derivatives. Now, the generalized Bessel function is defined as the projection of $D_k$ onto the space of $W$-invariant functions: 
\begin{equation}\label{Def1}
D_k^W(x,y) := \frac{1}{|W|} \sum_{w \in W} D_k(x,wy) =  \frac{1}{|W|} \sum_{w \in W} D_k(wx,y)
\end{equation}
The root system we shall be dealing with in this paper is of type $B_2$ and is defined by: 
\begin{equation*}
R = \{\pm e_1, \pm e_2, \pm e_1 \pm e_2\} \subset V = \mathbb{R}^2, 
\end{equation*}
The corresponding reflection group is isomorphic to the dihedral group of order eight:
\begin{equation*}
W =  \left(\mathbb{Z}_2\right)^2 \rtimes S_2 
\end{equation*}
and the action of $W$ on $R$ splits the latter into two orbits
\begin{equation*}
\{\pm (e_1 \pm e_2)\} \quad \{\pm e_1, \pm e_2\}
\end{equation*}
to which we attach two multiplicity values $k_1, k_2, \in \mathbb{R}_+^2$ respectively. In this case, it was proved in \cite{Demni} that: 
\begin{equation}\label{1}
D_k^W(x,y)= c_k \int_{-1}^1\int_{-1}^1\mathcal{I}_{\gamma-1/2}\left(\sqrt{\frac{Z_{x,y}(u,v)}{2}}\right)(1-u^2)^{k_1-1}(1-v^2)^{k_2-1}dudv
\end{equation}
where $\gamma := k_1+k_2$, $x = (x_1,x_2), y = (y_1,y_2)$, 
\begin{equation*}
c_k=\frac{\Gamma(k_1+1/2)\Gamma(k_2+1/2)}{\pi\Gamma(k_1)\Gamma(k_2)},
\end{equation*}
is the normalizing constant, 
\begin{equation*}
\mathcal{I}_{ \nu}(t)=\Gamma(\nu+1)\sum_{n=0}^{+\infty}\frac{(t/2)^{2n}}{n!\Gamma(n+\nu+1)}, \quad t \in \mathbb{R}.
\end{equation*}
is the normalized modified Bessel function of the first kind and of order $\nu$, and 
\begin{equation}\label{Argument}
Z_{x,y}(u,v) := (x_1^2+x_2^2)(y_1^2+y_2^2)+u(x_1^2-x_2^2)(y_1^2-y_2^2)+4vx_1x_2y_1y_2.
\end{equation}
Actually, \eqref{1} is a particular instance of a more general formula valid for integer values of $\gamma$, nevertheless it can extended to complex values with non negative real part using Carleson criteria as explained at the end of \cite{Dem1}. One can also see that \eqref{1} holds for non negative values of $\gamma$ by applying Lemma 8.5.2 in \cite{Kob-Man} to the Fourier-Gegenbauer series (3) in \cite{Demni} (with the substitutions $\nu \rightarrow \gamma, p=2$). 
 
\section{Laplace-type integral representation of the GBF of type $B_2$}
With the help of the notations introduced in the previous section, our main result may be stated as follows:
 \begin{thm}\label{th1}
If $\gamma > 1/2$, then the GBF of type $B_2$ admits the following Laplace-type integral representation:
\begin{equation}\label{B2}
D_k^W(x,y)= \int_{ \mathbb{R}^2}e^{\langle x,z\rangle}H_k(y,z)\;dz
\end{equation}
where 
\begin{multline*}
H_k(y,z) := \frac{(2\gamma-1)c_k}{\pi}
\int_{E_{y,z}} \Big( (y_1^2-y_2^2)^2(1-u^2)+4y_1^2y_2^2(1-v^2)\Big)^{-\gamma+1}
\\ \Big\{(y_1^2+y_2^2-(z_1^2+z_2^2))^2-(z_1^2-z_2^2-u(y_1^2-y_2^2))^2-4(z_1z_2-v y_1y_2)^2\Big\}^{\gamma-3/2} \qquad (1-u^2)^{k_1-1}(1-v^2)^{k_2-1}du\;dv
 \end{multline*}
and
\begin{eqnarray*}
E_{y,z} := \Big\{(u,v)\in[-1,1]^2;\;(z_1^2-z_2^2-u(y_1^2-y_2^2))^2+4(z_1z_2-v y_1y_2)^2\leq (y_1^2+y_2^2-(z_1^2+z_2^2))^2\Big\}.
\end{eqnarray*}
\end{thm}

\begin{proof}
Without loss of generality, we may assume $y \neq (0,0)$ since $D_k^W(x,(0,0)) = 1$ for all $x \in \mathbb{R}^2$.  Now, set\footnote{We ommit the dependence of $a,b,c$ on $(y,u,v)$ for sake of simplicity.}
\begin{eqnarray*}
a &:=& \sqrt{\frac{y_1^2+y_2^2+u(y_1^2-y_2^2)}{2}}, 
\\ b &:= & \frac{ \sqrt{ (y_1^2-y_2^2)^2(1-u^2)+4y_1^2y_2^2(1-v^2)}}{  y_1^2+y_2^2+u(y_1^2-y_2^2)} \\
c & := & \frac{2v  y_1y_2}{y_1^2+y_2^2+u(y_1^2-y_2^2)}.
\end{eqnarray*}
Then, we can check that 
\begin{equation*}
\sqrt{\frac{Z_{x,y}(u,v)}{2}} = a\sqrt{(x_1 +cx_2)^2+b^2x_2^2},
\end{equation*}
which shows that for fixed $y$, the argument of the modified Bessel function displayed in \eqref{1} may be written as the Euclidean norm of a planar vector. Furthermore, we shall need the following lemma:
\begin{lem}\label{l1} 
For any $\nu>0$ and any $z=(z_1,z_2) \in \mathbb{R}^2$, we have
\begin{eqnarray*}\label{bessel} \mathcal{I}_{\nu }(\|z\|)=\frac{  \nu }{\pi } \int_{ \{\|y\|\leq 1\} } e^{  \langle z,y\rangle }(1-\|y\|^2)^{\nu-1}dy,
\end{eqnarray*}
where $\|z\|=\sqrt{z_1^2+z_2^2}$.
\end{lem}
\begin{proof}
Using the invariance of the Lebesgue measure under rotations and a polar coordinates variable change, we write:
\begin{align*}
\int_{\{\|y\|\leq 1\} } e^{\langle z,y\rangle }(1-\|y\|^2)^{\nu-1}dy & = \int_{0}^{1}\int_{0}^{2\pi} e^{\|z\|s\cos \theta}  s (1-s^2)^{\nu-1}d\theta ds
\end{align*}
Integrating with respect to $\theta$ and appealing to the integral representation (\cite{Wat})
\begin{equation*}
\mathcal{I}_{\nu}(t) = \frac{\Gamma(\nu+1)}{2\sqrt{\pi} \Gamma(\nu+1/2)}\int_{0}^{2\pi} e^{t\cos \theta}\sin^{2\nu}\theta \;d\theta, 
\end{equation*}
we get: 
\begin{equation*}
\int_{\{\|y\|\leq 1\} } e^{\langle z,y\rangle }(1-\|y\|^2)^{\nu-1}dy=  2\pi\int_{0}^{1}\mathcal{I}_0(\|z\|s) s (1-s^2)^{\nu-1}ds.
\end{equation*}
Expanding the modified Bessel function $\mathcal{I}_0$ and integrating termwise finishes the proof.  
\end{proof}

Using lemma \ref{l1} followed by an affine change of variables, we readily derive
\begin{align} \nonumber
\mathcal{I}_{\gamma-1/2}\left(\sqrt{\frac{Z_{x,y}(u,v)}{2}}\right) & = \frac{ \gamma-1/2}{\pi } \int_{ \{\|z\|\leq 1\}} e^{ ax_1z_1+ax_2( cz_1+b z_2 )}(1-\|z\|^2)^{\gamma-3/2}dz
 \\& = \frac{ \gamma-1/2}{\pi } \int_{ \{b^2z_1^2+(z_2-cz_1)^2\leq a^2b^2 \} }e^{  \langle x,z\rangle }\;a^{-2\gamma+1}b^{-2\gamma+2}(a^2b^2-b^2z_1^2-(z_2-cz_1)^2)^{\gamma-3/2}dz \label{f1}.
\end{align}
Besides,
\begin{eqnarray*}
&&a^{-2\gamma+1}b^{-2\gamma+2}(a^2b^2-b^2z_1^2-(z_2-cz_1)^2)^{\gamma-3/2}=2\Big( (y_1^2-y_2^2)^2(1-u^2)+4y_1^2y_2^2(1-v^2)\Big)^{-\gamma+1}
\\ &&\Big\{(y_1^2+y_2^2-(z_1^2+z_2^2))^2-(z_1^2-z_2^2-u(y_1^2-y_2^2))^2-4(z_1z_2-v y_1y_2)^2\Big\}^{\gamma-3/2},
\end{eqnarray*}
which leads to the desired result.
\end{proof}
 
From Corollary 3.3 in \cite{dJ}, the support of $z \mapsto H(y,z)$ lies in the convex hull co$(y)$ of the orbit of $y$ under the action of $W$. This geometrical fact is directly proved in the following proposition.  
\begin{prop} 
For any $y \neq (0,0)$, the density $z\mapsto H(y,z)$ is supported in:
\begin{equation*}
\textrm{co}(y) = \{z=(z_1,z_2);\; |z_1|+|z_2|\leq |y_1|+|y_2|;\: \max( |z_1|,|z_2|)\leq \max( |y_1|,|y_2|) \}.
\end{equation*}
\end{prop}
\begin{proof}
From the definition of $H_k(y,z)$, it is sufficient to prove that for fixed $y \neq 0$, $E_{y,z}\neq \emptyset \Rightarrow z\in co(y)$. To this end, let $y_C$ and $z_C$ be the representatives of $y$ and $z$ in $\overline{C}$ and recall from \cite{Kos}, Lemma 3.3, that 
\begin{equation*}
z \in \textrm{co}(y) \quad \Leftrightarrow \quad y_C - z_C \in \sum_{\alpha \in S} \mathbb{R}_+ \alpha. 
\end{equation*}
This characterization of the convex hull of $Wy$ together with the obvious fact that 
\begin{equation*}
y_C = \left(\max(|y_1|, |y_2|), \min(|y_1|, |y_2|)\right), \quad z_C = \left(\max(|z_1|, |z_2|), \min(|z_1|, |z_2|)\right), 
\end{equation*}
lead to the above description of co$(y)$. Now, Assume $E_{y,z}\neq \emptyset$ and recall the definition:
\begin{equation*}
E_{y,z}=\Big\{(u,v)\in[-1,1]^2;\;b^2z_1^2+(z_2-cz_1)^2\leq a^2b^2 \Big\}.
\end{equation*}
Then, we readily see that
\begin{equation*}
|z_1|\leq a = \sqrt{\frac{ y_1^2+y_2^2+u(y_1^2-y_2^2)}{2}}\leq \max( |y_1|,|y_2|), \quad.
\end{equation*}
We can also rewrite $E_{y,z}$ as 
\begin{equation*}
E_{y,z}=\Big\{(u,v)\in[-1,1]^2;\; (b^2+c^2)\left(z_1 - \frac{c}{b^2+c^2} z_2\right)^2 + \frac{b^2}{(b^2+c^2)} z_2^2\leq a^2b^2 \Big\}
\end{equation*}
whence we deduce the inequality 
\begin{equation*}
|z_2|\leq a\sqrt{b^2+c^2} = \sqrt{\frac{ y_1^2+y_2^2-u(y_1^2-y_2^2)}{2}}\leq \max( |y_1|,|y_2|).
\end{equation*}
In particular, $z_1^2+z_2^2\leq y_1^2+y_2^2.$
Next, write 
\begin{eqnarray*}
&& (y_1^2+y_2^2-(z_1^2+z_2^2))^2-(z_1^2-z_2^2-u(y_1^2-y_2^2))^2-4(z_1z_2-v y_1y_2)^2\\&&\qquad\qquad=
 \Big(y_1^2+y_2^2+2vy_1y_2 - (z_1-z_2)^2\Big)\Big(y_1^2+y_2^2-2vy_1y_2- (z_1+z_2)^2\Big)\\&& \qquad\qquad\qquad\qquad\qquad\qquad\qquad- (z_1^2-z_2^2-u(y_1^2-y_2^2))^2,
  \end{eqnarray*}
to see that
\begin{equation*}
\Big(y_1^2+y_2^2+2vy_1y_2 - (z_1-z_2)^2\Big) \Big(y_1^2+y_2^2-2vy_1y_2- (z_1+z_2)^2\Big)\geq 0.
\end{equation*}
But
\begin{equation*}
\Big(y_1^2+y_2^2+2vy_1y_2 - (z_1-z_2)^2\Big)+ \Big(y_1^2+y_2^2-2vy_1y_2- (z_1+z_2)^2\Big)  = 2(y_1^2+y_2^2-(z_1^2+z_2^2))\geq 0 
\end{equation*}
therefore 
 \begin{equation*}
 y_1^2+y_2^2+2vy_1y_2 - (z_1-z_2)^2\geq 0; \quad y_1^2+y_2^2-2vy_1y_2- (z_1+z_2)^2 \geq 0. 
 \end{equation*}
Since $|v| \leq 1$, then 
\begin{equation*}
|z_1-z_2|\leq  |y_1|+|y_2| ;\qquad  |z_1+z_2|\leq  |y_1|+ |y_2|
\end{equation*}
whence it follows that $|z_1|+|z_2|\leq  |y_1|+|y_2|$ and the proposition is proved. 
\end{proof}

\begin{remark}
When $k_0 = k_1 = 1$ and $y_1 > y_2 > 0$ (or equivalently $y \in C$), the measure $H_1(y,z)dz$ reduces, up to a multiplicative factor, to the Duistermaat-Heckman measure associated with the group $W = (\mathbb{Z}_2)^2 \rtimes S_2$ (\cite{BBO}):   
\begin{align*}
H_1(y,z) &= \frac{3}{4\pi} \int_{E_{y,z}} \frac{\sqrt{(y_1^2+y_2^2-(z_1^2+z_2^2))^2-(z_1^2-z_2^2-u(y_1^2-y_2^2))^2-4(z_1z_2-v y_1y_2)^2}}{(y_1^2-y_2^2)^2(1-u^2)+4y_1^2y_2^2(1-v^2)}
du\;dv
\\& = \frac{3}{4\pi(y_1^2-y_2^2)y_1y_2} \int_{F_{y,z}} \frac{\sqrt{(y_1^2+y_2^2-(z_1^2+z_2^2))^2-(z_1^2-z_2^2-u)^2- (2z_1z_2-v)^2}}{(y_1^2+y_2^2)^2 - u^2 -  v^2}du dv
\\& = \frac{1}{2(y_1^2-y_2^2)y_1y_2}m_{\textrm{DH}}(y,dz)
\end{align*}
where 
\begin{eqnarray*}
F_{y,z} := \Big\{|u| < y_1^2-y_2^2 ,|v| < 2y_1 y_2 ;\;(z_1^2-z_2^2-u)^2+ (2z_1z_2-v)^2\leq (y_1^2+y_2^2-(z_1^2+z_2^2))^2\Big\}.
\end{eqnarray*}

Moreover, the former is the conditional distribution of the value of a planar Brownian motion at a fixed time $T > 0$ given that the value of the generalized Pitman transform of this process at time $T$ is $y$, while the latter is the push forward of the Lebesgue measure of some (string) polytope $\mathcal{P}$ in $\mathbb{R}^4$ under an affine map. In suitable coordinates, $\mathcal{P}$ transforms into the Gelfand-Tsetlin pattern of $sp(4,\mathbb{C})$ at $y$ defined by (see e.g. \cite{BZ}, \cite{Ok}):
\begin{equation*}
\textrm{GT}_{\textrm{sp}(4,\mathbb{C})} := \{(w_1, w_2, t), \, y_1 \geq w_1 \geq y_2 \geq w_2 \geq 0, \, w_1 \geq t \geq w_2\}. 
 \end{equation*}
Note that for the $A$-type root system or corresponding to the Lie algebra sl$(N, \mathbb{C})$, the Gelfand-Tsetlin pattern shows up directly in the inductive construction of the GBF since a similar one already exists for Jack polynomials (\cite{Amri0}).\end{remark}

\section{Laplace-type integral representation of the Dunkl kernel of type $B_2$}
Using Theorem \ref{th1}, we shall derive a Laplace-type integral representation for the Dunkl kernel $D_k$ associated with the root system of type $B_2$. To this end, we appeal to the so-called Shift principle which we briefly outline (see \cite{Dunkl1}, Prop. 1.4). Let $\mathcal{C}$ be an orbit of the action of $W$ on $R$ and let $k$ be a nonnegative multiplicity function. Then
\begin{equation*}
k' := k+ {\bf 1}_{\mathcal{C}}
\end{equation*}
is the shifted multiplicity function on $\mathcal{C}$ and the following holds: 
\begin{equation*}
\sum_{w \in W} \chi_{\mathcal{C}}(w)D_k(x,wy)  = d_{\mathcal{C}}(k) p_{\mathcal{C}}(x)p_{\mathcal{C}}(y) D_{k'}^W(x,y), \quad x, y \in C
\end{equation*}
where 
\begin{equation*}
p_{\mathcal{C}}(x) := \prod_{\alpha \in \mathcal{C} \cap R_+} \langle \alpha, x \rangle, 
\end{equation*}
is the alternating polynomial associated with $\mathcal{C}$, $\chi_{\mathcal{C}}$ is its linear character defined by:      
\begin{equation*}
p_{\mathcal{C}}(wx) = \chi_{\mathcal{C}}(w) p_{\mathcal{C}}(x), \quad w \in W, 
\end{equation*}
and $d_k$ is the normalizing constant given by 
\begin{equation*}
d_ {\mathcal{C}}(k) := \frac{|W|}{p_{\mathcal{C}}(T_1,\dots, T_N) p_{\mathcal{C}}}.
\end{equation*}
This principle was already used in \cite{Amri1} to derive the Laplace-type integral representation of the Dunkl kernel associated with the $A_2$-type root system and in \cite{DDY} in a similar fashion to get an integral representation of $D_k$ associated with the $B_2$-type root system. Nonetheless, the representation we prove below improves the one derived in \cite{DDY} since it is of Laplace-type. In order to state it, we shall apply the shift principle to the orbits
\begin{equation*}
 \mathcal{C}_1 = \{ \pm (e_1-e_2), \pm (e_1+e_2)\}, \quad \mathcal{C}_2 = \{ \pm e_1, \pm e_2\},
\end{equation*}
to which assigned the multiplicity values $k_1, k_2$ respectively, and to the whole root system $R$ as well. Consequently, the corresponding alternating polynomials are given by 
\begin{equation*}
S_1(x) := p_{\mathcal{C}_1}(x) = x_1^2-x_2^2, \quad S_2(x) := p_{\mathcal{C}_2}(x) = x_1x_2, \quad S_3(x) := p_{R}(x) = x_1x_2(x_1^2-x_2^2),
\end{equation*}
whence we easily derive the following expressions\footnote{See \cite{DDO}, p. 249, for the expression of $d_3(k)$.}: 
\begin{eqnarray*}
 d_1(k):=  d_ {\mathcal{C}_1}(k) & = & \frac{2}{(2k_1+1)(2\gamma+1)}, \\ 
 d_2(k) := d_ {\mathcal{C}_2}(k) & = & \frac{8}{(2k_2+1)(2\gamma+1)} ,\\
 d_3(k) := d_ {R}(k) & = & \frac{1}{2(2k_1+1)(2k_2+1)(2\gamma+1)(2\gamma+3)}.
\end{eqnarray*}

With these findings, we prove the following:

 \begin{prop}\label{th2}
The Dunkl kernel and the GBF of type $B_2$ are interrelated  by
\begin{align}\label{E2}
2y_1\;D_k(x,y) & = (T_1+y_1)\left\{\frac{1}{4}\left(d_1(k)S _1(x)S_1 (y)D_{(k_1+1,k_2)}^W(x,y) \right. \right.
\\&\nonumber \left.\left.+d_2(k)S_2(x)S_2(y)D_{(k_1,k_2+1)}^W(x,y)\ +d_3(k)S_3(x)S_3(y) D_{(k_1+1,k_2+1)}^W(x,y)\right) + 2D_{k}^W(x,y)\right\},
 \end{align}
 where we write $D_k^W = D_{(k_1,k_2)}^W$ in order to emphasize the shift action of $k = (k_1,k_2)$ and $T_1$ is the Dunkl derivative in the direction $e_1$ acting on $x$.  
\end{prop}
\begin{proof}
Recall that the Weyl group $W = \left(\mathbb{Z}_2\right)^2 \rtimes S_2 $ consists of the eight elements below:
\begin{eqnarray*}
   \sigma_1=\begin{pmatrix}
     -1 & 0 \\
     0 & 1\\
    \end{pmatrix},
    \; \sigma_2=\begin{pmatrix}
     1 & 0 \\
     0 & -1\\
    \end{pmatrix},
     \; \sigma_3=\begin{pmatrix}
     0 &1 \\
      1 & 0\\
    \end{pmatrix},
    \sigma_4=\begin{pmatrix}
     0 & -1 \\
     -1 & 0\\
    \end{pmatrix},
    \\ \; r=\begin{pmatrix}
    0 & 1 \\
     -1 & 0\\
      \end{pmatrix},
       \; r^2=\begin{pmatrix}
   -1 &0 \\
     0 & -1\\
      \end{pmatrix},
       r^3=\begin{pmatrix}
    0 & -1 \\
     1 & 0\\
      \end{pmatrix},
       \; id=\begin{pmatrix}
    1 & 0 \\
     0 & 1\\
      \end{pmatrix}.
    \end{eqnarray*}

Applying the Shift principle to $\mathcal{C}_1, \mathcal{C}_2, R$ respectively, we derive the following identities:
     \begin{eqnarray*}
    && D_k(x,\sigma_1y)+D_k(x,\sigma_2y)-D_k(x,\sigma_3y)-D_k(x,\sigma_4y)-D_k(x,ry)
    +D_k(x,r^2y)\\&&-D_k(x,r^3y)+D_k(x, y)=d_1(k)S_1(x)S_1(y)D_{(k_1+1,k_2)}^W(x,y),
\end{eqnarray*}
   \begin{eqnarray*}
   && - D_k(x,\sigma_1y) -D_k(x,\sigma_2y) + D_k(x,\sigma_3y) + D_k(x,\sigma_4y) - D_k(x,ry)
    +D_k(x,r^2y)\\&&-D_k(x,r^3y)+D_k(x, y) =d_2(k)S_2(x)S_2(y)D_{(k_1,k_2+1)}^W(x,y),
    \end{eqnarray*}
  and
   \begin{eqnarray*}
&&-D_k(x,\sigma_1y)-D_k(x,\sigma_2y)-D_k(x,\sigma_3y)-D_k(x,\sigma_4y)+D_k(x,ry)
    +D_k(x,r^2y)\\&&+D_k(x,r^3y)+D_k(x,y)=d_3(k)S_3 (x)S _3(y)D_{(k_1+1,k+1)}^W(x,y),
      \end{eqnarray*}
Together with \eqref{Def1}, they lead to     
\begin{eqnarray*}
 &&D_k(x,y)+D_k(x, -y) =\frac{1}{4}\Big(d_1(k)S _1(x)S_1 (y)D_{(k_1+1,k_2)}^W(x,y) \\&&+d_2(k)S_2(x)S_2(y)D_{(k_1,k_2+1)}^W(x,y)  +d_3(k)S_3(x)S_3(y) D_{(k_1+1,k_2+1)}^W(x,y)\Big)
 + 2D_{k}^W(x,y)
 \end{eqnarray*}
Since $T_1D_k(\cdot, y)(x) = y_1D_k(x,y)$, then the proposition follows.
 \end{proof}
   
With regard to proposition \ref{th2}, we need to work out the right hand side of \eqref{E2}. The issue of our computations is recorded in the corollary below: 
\begin{corollary}
If $\gamma = k_0+k_1 > 1/2$, then the Dunkl kernel of type $B_2$ admits the following Laplace-type integral representation: 
\begin{equation*}
2y_1D_k(x,y) = \int_{\textrm{co}(y)} e^{\langle x,z \rangle} L_k(y,z) dz 
\end{equation*}
where 
\begin{multline*}
L_k(y,z) := \int_{E_{y,z}} \Bigg\{2(z_1+y_1)(4u +4v+uv+4) +  2\frac{S_2(y)}{(ab)^2}(z_2-cz_1)(u+4)(1-v^2)  \\ +\frac{S_1(y)}{(ab)^2} ((b^2+c^2)z_1 - cz_2)(v+4)(1-u^2)\Bigg\} 
 \frac{(a^2b^2-b^2z_1^2-(z_2-cz_1)^2)^{\gamma-3/2}}{a^{2\gamma-1}b^{2\gamma-2}} (1-u^2)^{k_1-1}(1-v^2)^{k_2-1}\;dz\;du dv.
 \end{multline*}
 \end{corollary}
\begin{proof}
Write \eqref{Argument} as
\begin{equation*}
Z_{x,y}(u,v) = (x_1^2+x_2^2)(y_1^2+y_2^2) + u S_1(x)S_1(y) + 4vS_2(x)S_2(y).
\end{equation*}
Then, the differentiation rule 
\begin{equation}\label{DiffRule}
\frac{d}{dt}\mathcal{I}_{\nu}\left(\sqrt{\frac{t}{2}}\right) = \frac{1}{8(\nu+1)} \mathcal{I}_{\nu+1}\left(\sqrt{\frac{t}{2}}\right),
\end{equation}
and \eqref{1} yield: 
\begin{align*}
d_1(k)S _1(x)S_1 (y)D_{(k_1+1,k_2)}^W(x,y) & = \frac{4c_k}{k_1}\int_{-1}^1\int_{-1}^1\frac{\partial}{\partial u}\left \{ \mathcal{I}_{\gamma-1/2}\left(\sqrt{\frac{Z_{x,y}(u,v)}{2}}\right) \right\}  (1-u^2)^{k_1 }(1-v^2)^{k_2-1}du dv
\\&= 8 c_k \int_{-1}^1\int_{-1}^1 \mathcal{I}_{\gamma-1/2}\left(\sqrt{\frac{Z_{x,y}(u,v)}{2}}\right)   u(1-u^2)^{k_1-1}(1-v^2)^{k_2-1}\;du dv.
\end{align*}
Similarly, 
\begin{align*}
d_2(k)S _2(x)S_2 (y)D_{(k_1 ,k_2+1)}^W(x,y) &= \frac{4c_k}{k_2}\int_{-1}^1\int_{-1}^1\frac{\partial}{\partial v}\left \{ \mathcal{I}_{\gamma-1/2}\left(\sqrt{\frac{Z_{x,y}(u,v)}{2}}\right) \right\}  (1-u^2)^{k_1-1}(1-v^2)^{k_2}du dv
\\& = 8 c_k\int_{-1}^1\int_{-1}^1 \mathcal{I}_{\gamma-1/2}\left(\sqrt{\frac{Z_{x,y}(u,v)}{2}}\right) v(1-v^2)^{k_2-1}  (1-u^2)^{k_1-1}\;du dv,
\end{align*}
 and
 \begin{align*}
d_3(k)S _3(x)S_3 (y)D_{(k_1+1 ,k_2+1)}^W(x,y) &= \frac{c_k}{2k_1k_2 }\int_{-1}^1\int_{-1}^1\frac{\partial^2}{\partial u \partial v}\left \{ \mathcal{I}_{\gamma-1/2}\left(\sqrt{\frac{Z_{x,y}(u,v)}{2}}\right) \right\}  (1-u^2)^{k_1}(1-v^2)^{k_2}du dv
\\& = 2c_k \int_{-1}^1\int_{-1}^1 \mathcal{I}_{\gamma-1/2}\left(\sqrt{\frac{Z_{x,y}(u,v)}{2}}\right)uv(1-u^2)^{k_1-1}(1-v^2)^{k_2-1}\;du dv.
\end{align*}
Altogether yield
\begin{eqnarray*}
 &&  \frac{1}{4}\Big(d_1(k)S _1(x)S_1 (y)D_{(k_1+1,k_2)}^W(x,y)  +d_2(k)S_2(x)S_2(y)D_{(k_1,k_2+1)}^W(x,y)  \\&&\qquad\qquad\qquad
 \qquad\qquad\qquad\qquad+d_3(k)S_3(x)S_3(y) D_{(k_1+1,k_2+1)}^W(x,y)\Big)
 + 2D_{k}^W(x,y)\\&=&
 \frac{c_k}{2} \int_{-1}^1\int_{-1}^1 \mathcal{I}_{\gamma-1/2}\left(\sqrt{\frac{Z_{x,y}(u,v)}{2}}\right)
 (4u+4v+uv+4)(1-u^2)^{k_1-1}(1-v^2)^{k_2-1}\;du dv
 \\&=&\frac{ (2\gamma-1)c_k}{4\pi }\int_{-1}^1\int_{-1}^1  \int_{ \{\|z\|\leq 1\} }
   e^{ ax_1z_1+ax_2( cz_1+b z_2 )}(1-\|z\|^2)^{\gamma-3/2}(4u+4v+uv+4)\\&&\qquad\qquad\qquad\qquad\qquad\qquad(1-u^2)^{k_1-1}(1-v^2)^{k_2-1}\;dz\;du dv.
 \end{eqnarray*}

Next, recall the expression of the Dunkl derivative $T_1$ in the direction $e_1$ acting on $x$: 
\begin{align*}
T_1f(x) & = \partial _1f(x) + \sum_{\alpha \in R_+} k(\alpha) \langle \alpha, e_1 \rangle \frac{f(x) - f(\sigma_{\alpha} x)}{\langle \alpha, x \rangle} 
\\& = \partial _1f(x) + k_2 \frac{f(x) - f(\sigma_1 x)}{x_1} + k_1 \left\{\frac{f(x) - f(\sigma_3 x)}{x_1-x_2} + \frac{f(x) - f(\sigma_4 x)}{x_1+x_2}\right\}. 
\end{align*}
Then, the invariance of $D_k^W$ under the action of $W$ entails:  
\begin{multline}\label{Equat0}
 T_1 \Bigg\{\frac{1}{4}\Big(d_1(k)S _1(x)S_1 (y)D_{(k_1+1,k_2)}^W(x,y) +d_2(k)S_2(x)S_2(y)D_{(k_1,k_2+1)}^W(x,y)  \\ + d_3(k)S_3(x)S_3(y) D_{(k_1+1,k_2+1)}^W(x,y)\Big)  
  + 2D_{k}^W(x,y)\Bigg\}  = \frac{ (2\gamma-1)c_k}{4\pi}  
  \int_{-1}^1\int_{-1}^1  \int_{ \{\|z\|\leq 1\} }
   az_1 e^{ ax_1z_1+ax_2( cz_1+b z_2 )} \\ (1-\|z\|^2)^{\gamma-3/2}  (4u+4v+uv+4) (1-u^2)^{k_1-1}(1-v^2)^{k_2-1}\;dz\;du dv
+ I_2(x,y) 
\end{multline}
where we set 
 \begin{multline*}
 I_2(x,y) := d_1(k)S_1(y)(k_1x_1)D_{(k_1+1,k_2)}^W(x,y)+ d_2(k)S_2(y)\frac{k_2x_2}{2}D_{(k_1,k_2+1)}^W(x,y)
 \\ +d_3(k)S_3(y)S_1(x)\frac{k_2x_2}{2}D_{(k_1+1,k_2+1)}^W(x,y) + d_3(k)S_3(y)S_2(x)(k_1x_1)D_{(k_1+1,k_2+1)}^W(x,y).
 \end{multline*}
 
Now, it remains to put $I_2(x,y)$ in form of  (\ref{f1}).  To proceed, we first use Lemma \ref{l1} and integration by parts to derive
\begin{align}\label{EqInter1}
x_2\mathcal{I}_{\gamma+1/2}\left(\sqrt{\frac{Z_{x,y}(u,v)}{2}}\right) & =\frac{ 2\gamma+1}{2\pi } \int_{ \{\|z\|\leq 1\} } x_2 e^{ ax_1z_1+ax_2( cz_1+b z_2 )}(1-\|z\|^2)^{\gamma-1/2}dz  \nonumber
  \\&=\frac{ (2\gamma+1)(2\gamma-1)}{2\pi } \int_{ \{\|z\|\leq 1\} } \frac{z_2}{ab} e^{ax_1z_1+ax_2( cz_1+b z_2 )}(1-\|z\|^2)^{\gamma-3/2}dz
\end{align}
and 
\begin{align}\label{EqInter2}
 x_1\mathcal{I}_{\gamma+1/2}\left(\sqrt{\frac{Z_{x,y}(u,v)}{2}}\right) & = \frac{(2\gamma+1)(2\gamma-1)}{2\pi } \int_{ \{\|z\|\leq 1\} }
 \left(\frac{z_1}{a}-\frac{cz_2}{ab}\right) e^{ ax_1z_1+ax_2( cz_1+b z_2 )} (1-\|z\|^2)^{\gamma-1/2}dz.
 \end{align}

It follows that  
\begin{multline}\label{Equat1}
\frac{k_2x_2}{2}d_2(k)S_2(y)D_{(k_1,k_2+1)}^W(x,y) =  \frac{ (2\gamma-1) c_k}{\pi} S_2(y) \int_{-1}^1\int_{-1}^1 \int_{ \{\|z\|\leq 1\} } \frac{z_2}{ab}  e^{ ax_1z_1+ax_2( cz_1+b z_2 )} \\
(1-\|z\|^2)^{\gamma-3/2}(1-u^2)^{k_1-1}(1-v^2)^{k_2}\;dz\;du dv
\end{multline}
and 
\begin{multline}\label{Equat2}
 d_1(k)S_1(y)(k_1x_1)D_{(k_1+1,k_2)}^W(x,y) = \frac{ (2\gamma-1) c_k}{2\pi} S_1(y) \int_{-1}^1\int_{-1}^1 \int_{ \{\|z\|\leq 1\} } \left(\frac{z_1}{a}-\frac{cz_2}{ab}\right) e^{ ax_1z_1+ax_2( cz_1+b z_2 )}
 \\ (1-\|z\|^2)^{\gamma-3/2} (1-u^2)^{k_1}(1-v^2)^{k_2-1}\;dz\;du dv.
\end{multline}
Besides, the differentiation rule \eqref{DiffRule}, \eqref{EqInter1} and \eqref{EqInter2} yield
\begin{multline}\label{Equat3}
d_3(k)(k_1x_1)S _2(x)S_3 (y)D_{(k_1+1,k_2+1)}^W(x,y) = \frac{d_3(k)(2k_1+1)(2k_2+1)(2\gamma+3)c_k}{4k_2} S_1(y)
\\ \int_{-1}^1\int_{-1}^1 x_1\frac{\partial}{\partial v}\left \{ \mathcal{I}_{\gamma+1/2}\left(\sqrt{\frac{Z_{x,y}(u,v)}{2}}\right) \right\}  (1-u^2)^{k_1 }(1-v^2)^{k_2}du dv  
\\ = \frac{d_3(k)(2k_1+1)(2k_2+1)(2\gamma+3)c_kS_1(y)}{2} \int_{-1}^1\int_{-1}^1x_1 \mathcal{I}_{\gamma+1/2}\left(\sqrt{\frac{Z_{x,y}(u,v)}{2}}\right) v(1-v^2)^{k_2-1}  (1-u^2)^{k_1 }\;du dv 
\\=\frac{(2\gamma-1)c_k}{8\pi } S_1(y) \int_{-1}^1\int_{-1}^1 \int_{ \{\|z\|\leq 1\} } \left(\frac{z_1}{a}-\frac{cz_2}{ab}\right) e^{ ax_1z_1+ax_2( cz_1+b z_2 )}
 \\(1-\|z\|^2)^{\gamma-1/2} v(1-v^2)^{k_2-1}  (1-u^2)^{k_1 }\;dz \;du dv
\end{multline}
and
\begin{multline}\label{Equat4}
d_3(k)\frac{k_2x_2}{2}S _1(x)S_3 (y)D_{(k_1+1,k_2+1)}^W(x,y) =\frac{d_3(k)(2k_1+1)(2k_2+1)(2\gamma+3)c_k}{2k_1} S_2(y)
\\ \int_{-1}^1\int_{-1}^1x_2\frac{\partial}{\partial u}\left \{ \mathcal{I}_{\gamma+1/2}\left(\sqrt{\frac{Z_{x,y}(u,v)}{2}}\right) \right\} (1-u^2)^{k_1 }(1-v^2)^{k_2}du dv
\\ = d_3(k)(2k_1+1)(2k_2+1)(2\gamma+3)S_2(y)\int_{-1}^1\int_{-1}^1x_2 \mathcal{I}_{\gamma+1/2}\left(\sqrt{\frac{Z_{x,y}(u,v)}{2}}\right) u(1-u^2)^{k_1-1}  (1-v^2)^{k_2 }\;du dv
\\= \frac{(2\gamma-1)c_k}{4\pi }S_2(y)\int_{-1}^1\int_{-1}^1 \int_{ \{\|z\|\leq 1\} } \frac{z_2}{ab}e^{ ax_1z_1+ax_2( cz_1+b z_2 )}(1-\|z\|^2)^{\gamma-3/2}
u(1-u^2)^{k_1-1}  (1-v^2)^{k_2 }\;dz\;du dv.
\end{multline}
Gathering \eqref{Equat1}, \eqref{Equat2}, \eqref{Equat3}, \eqref{Equat4}, we get 
 \begin{multline*}
 I_2(x,y)=  \frac{(2\gamma-1)c_k}{8\pi }\int_{-1}^1\int_{-1}^1 \int_{ \{\|z\|\leq 1\} } \Bigg\{2S_2(y)\frac{z_2}{ab}(u+4)(1-v^2) +S_1(y)\left(\frac{z_1}{a}-\frac{cz_2}{ab}\right)(v+4)(1-u^2)\Bigg\} 
 \\  e^{ ax_1z_1+ax_2( cz_1+b z_2 )}(1-\|z\|^2)^{\gamma-3/2}(1-u^2)^{k_1-1}  (1-v^2)^{k_2-1 }\;dz\;du dv.
 \end{multline*}
Keeping in mind \eqref{Equat0}, we get: 
 \begin{multline}
 T_1 \Bigg\{\frac{1}{4}\Big(d_1(k)S _1(x)S_1 (y)D_{(k_1+1,k_2)}^W(x,y) +d_2(k)S_2(x)S_2(y)D_{(k_1,k_2+1)}^W(x,y)\ + \\ d_3(k)S_3(x)S_3(y) D_{(k_1+1,k_2+1)}^W(x,y)\Big)
 + 2D_{k}^W(x,y)\Bigg\} 
  \\ = \frac{ (2\gamma-1)c_k}{8\pi} \int_{-1}^1\int_{-1}^1  \int_{ \{\|z\|\leq 1\} } \Bigg\{2az_1(4u+4v+uv+4) + 2S_2(y)\frac{z_2}{ab}(u+4)(1-v^2) +S_1(y)\left(\frac{z_1}{a}-\frac{cz_2}{ab}\right)(v+4)(1-u^2)\Bigg\}
  \\   e^{ ax_1z_1+ax_2( cz_1+b z_2 )}(1-\|z\|^2)^{\gamma-3/2}(1-u^2)^{k_1-1}(1-v^2)^{k_2-1}\;dz\;du dv.
 \end{multline}
Consequently, the previous proposition implies  
\begin{multline*}
2y_1 D_k(x,y) = \frac{ (2\gamma-1)c_k}{8\pi} \int_{-1}^1\int_{-1}^1  \int_{ \{\|z\|\leq 1\} } \Bigg\{2(az_1+y_1)(4u +4v+uv+4) +  2S_2(y)\frac{z_2}{ab}(u+4)(1-v^2)  \\ +S_1(y)\left(\frac{z_1}{a}-\frac{cz_2}{ab}\right)(v+4)(1-u^2)\Bigg\}
    e^{ ax_1z_1+ax_2( cz_1+b z_2 )} 
   (1-\|z\|^2)^{\gamma-3/2}(1-u^2)^{k_1-1}(1-v^2)^{k_2-1}\;dz\;du dv.
  \end{multline*}
Performing the variables change 
\begin{equation*}
az_1 \rightarrow z_1, \quad a(cz_1 + bz_2) \rightarrow z_2, 
\end{equation*}
we finally get the formula displayed in the corollary. 
 \end{proof}

\end{document}